\documentclass[10pt]{amsart}

\usepackage{amsmath}
\usepackage{amssymb}
\usepackage{bm}
\usepackage{graphicx}
\usepackage{psfrag}
\usepackage{color}
\usepackage{hyperref}
\hypersetup{colorlinks=true, linkcolor=blue, citecolor=magenta, urlcolor=wine}
\usepackage{url}
\usepackage{algpseudocode}
\usepackage{comment}
\usepackage{fancyhdr}
\usepackage{mathtools}
\usepackage{tikz-cd}
\usepackage{xy}
\input xy
\xyoption{all}
\usepackage{stmaryrd}
\usepackage{calrsfs}

\newtheorem*{maintheorem*}{Main Theorem}
\newtheorem{theorem}{Theorem}[section]
\newtheorem{prop}[theorem]{Proposition}
\newtheorem{conj}[theorem]{Conjecture}
\newtheorem{lem}[theorem]{Lemma}
\newtheorem{cor}[theorem]{Corollary}

\theoremstyle{definition}
\newtheorem{defn}[theorem]{Definition}

\newtheorem{ex}[theorem]{Example}

\numberwithin{equation}{section}

\newcommand{\cc}{\mathbb{C}}
\newcommand{\ff}{\mathbb{F}}
\newcommand{\nn}{\mathbb{N}}
\newcommand{\pp}{\mathbb{P}}
\newcommand{\qq}{\mathbb{Q}}
\newcommand{\rr}{\mathbb{R}}
\newcommand{\zz}{\mathbb{Z}}
\newcommand{\ch}{\text{char}}

\providecommand\ldb{\llbracket}
\providecommand\rdb{\rrbracket}

\newcommand{\gp}{\text{gp}}
\newcommand{\qf}{\text{qf}}

\newcommand{\ii}{\mathcal{Irr}}



\begin{document}
	
	\mbox{}
	\title{Hereditary atomicity in integral domains}
	
	\author{Jim Coykendall}
	\address{School of Mathematical and Statistical Sciences\\Clemson University\\Clemson, SC 29634}
	\email{jcoyken@clemson.edu}
	
	\author{Felix Gotti}
	\address{Department of Mathematics\\MIT\\Cambridge, MA 02139}
	\email{fgotti@mit.edu}
	
	\author{Richard Erwin Hasenauer}
	\address{Department of Mathematics\\Northeastern State University\\Tahlequah, OK 74464}
	\email{hasenaue@nsuok.edu}

\date{\today}

\begin{abstract}
	If every subring of an integral domain is atomic, we say that the latter is hereditarily atomic. In this paper, we study hereditarily atomic domains. First, we characterize when certain direct limits of Dedekind domains are Dedekind domains in terms of atomic overrings. Then we use this characterization to determine the fields that are hereditarily atomic. On the other hand, we investigate hereditary atomicity in the context of rings of polynomials and rings of Laurent polynomials, characterizing the fields and rings whose rings of polynomials and rings of Laurent polynomials, respectively, are hereditarily atomic. As a result, we obtain two classes of hereditarily atomic domains that cannot be embedded into any hereditarily atomic field. By contrast, we show that rings of power series are never hereditarily atomic. Finally, we make some progress on the still open question of whether every subring of a hereditarily atomic domain satisfies ACCP.
\end{abstract}

\bigskip
\maketitle

\medskip
\section{Introduction}
\label{sec:intro}
\smallskip

An integral domain $R$ is said to be atomic if every nonzero nonunit of $R$ factors into irreducibles. Given that many relevant classes of commutative rings, including Krull domains and Noetherian domains, are atomic, the property of being atomic has been largely investigated in commutative ring theory. Since the appearance of the paper~\cite{aG74}, where A. Grams studies ascending chains of principal ideals in atomic integral domains, special attention has been given to atomicity in connection to ascending chains of ideals. The study of atomicity was significantly stimulated in 1990 by the paper~\cite{AAZ90}, where D. D. Anderson, D. F. Anderson, and M. Zafrullah introduce the bounded and finite factorization properties, which are properties that refine the class of atomic domains based on the notion of factorizations. Noetherian and Krull domains naturally fit in the proposed refinement, as they satisfy the bounded and the finite factorization properties, respectively (see \cite[Proposition~2.2 and page~14]{AAZ90}, \cite[Theorem~2]{AM96}, and \cite[Theorem~5]{HK92}).
\smallskip

Most of the important ring-theoretical properties of integral domains, including being Krull and being Noetherian, are not inherited by subrings. In~\cite{rG70}, R. Gilmer investigates integral domains whose subrings are Noetherian, concluding with a full characterization. Here we adopt the same methodology, but we replace the property of being Noetherian by the property of being atomic. Accordingly, we say that an integral domain $R$ is hereditarily atomic if every subring of $R$ is atomic. It is worth noticing that the property of being atomic is way more subtle than that of being Noetherian. For instance, in contrast to the latter, the former is not preserved under localization~\cite[page~189]{rG84} and it does not transfer to polynomial rings~\cite{mR93}. So, unlike the Noetherian counterpart, a full characterization of the hereditarily atomic domains seems to be highly difficult to establish. Having said that, here we offer what may be the first step towards a full classification, and we hope that our results shed some light upon a better understanding of hereditary atomicity in integral domains.
\smallskip

In Section~\ref{sec:fields}, we give a full characterization of the fields that are hereditarily atomic. We prove that a field $F$ of characteristic zero is hereditarily atomic if and only if $F$ is an algebraic extension of $\qq$ and the integral closure of $\zz$ in $F$ is a Dedekind domain. We also establish a parallel characterization for fields of positive characteristic. In our way to establish these characterizations, directed systems of integral domains $(R_\gamma)_{\gamma \in \Gamma}$, where $R_\alpha \subseteq R_\beta$ is integral and $[\qf(R_\alpha) : \qf(R_\beta)]$ is finite when $\alpha \le \beta$, played an important role (here $\qf(R)$ denotes the quotient field of $R$); we call them integral directed systems. We have found atomic conditions under which the union of an integral directed system of Dedekind domains is a Dedekind domain; we observe that, unlike for Pr\"ufer domains, the directed union of Dedekind domains may not be a Dedekind domain.
\smallskip

In contrast to Section~\ref{sec:fields}, the primary purpose of Section~\ref{sec:polynomials} is to provide classes of hereditarily atomic integral domains whose quotient fields are not hereditarily atomic. To achieve this, we consider hereditary atomicity in rings of polynomials and rings of Laurent polynomials. First, we determine the fields whose rings of polynomials are hereditarily atomic. Then we characterize the rings whose rings of Laurent polynomials are hereditarily atomic: it turns out that such rings are precisely the algebraic extensions of finite fields. We also prove that no ring of power series can be hereditarily atomic.

In Section~\ref{sec:hereditary ACCP}, we study the condition of being hereditarily ACCP. As for atomicity, we say that an integral domain $R$ is hereditarily ACCP if every subring of $R$ satisfies the ascending chain condition on principal ideals (ACCP). We conclude the paper exploring the following still open question: Is every hereditarily atomic domain hereditarily ACCP? We conjecture that this is, in fact, the case. Then we prove that the conjecture actually holds for several subclasses of atomic domains, while profiling some general properties of a potential counterexample should it exist.

\bigskip
\section{Background}
\label{sec:background}
\smallskip

Following standard notation, we let $\zz$, $\qq$, $\rr$, and $\cc$ denote the set of integers, rational numbers, real numbers, and complex numbers, respectively. Also, we let $\nn$ and $\nn_0$ denote the sets of positive and nonnegative integers, respectively, while we let $\pp$ denote the set of primes. For $p \in \pp$ and $n \in \nn$, we let $\ff_{p^n}$ be the finite field of cardinality $p^n$. In addition, for $a,b \in \zz$ with $a \le b$, we let $\ldb a,b \rdb$ denote the set of integers between $a$ and $b$, that is, $\ldb a,b \rdb = \{n \in \zz \mid a \le n \le b\}$. Finally, for $S \subseteq \rr$ and $r \in \rr$, we set $S_{\ge r} = \{s \in S \mid s \ge r\}$ and $S_{> r} = \{s \in S \mid s > r\}$.
\smallskip

Let $R$ be an integral domain. Throughout this paper, $R^\ast$ and $R^\times$ stand for the multiplicative monoid and the group of units of $R$, respectively. Also, we let $\qf(R)$ stand for the quotient field of~$R$. If $R \subseteq S$ is an extension of integral domains, we let $\overline{R}_S$ denote the integral closure of $R$ in $S$, and we let $\overline{R}$ denote the integral closure of $R$ in $\qf(R)$. An \emph{overring} of~$R$ is an intermediate ring of the ring extension $R \subseteq \qf(R)$. We let $\ii(R)$ denote the set of all irreducibles of $R$, and $R$ is said to be \emph{antimatter} if $\ii(R)$ is empty. Following P.~M. Cohn~\cite{pC68}, we say that $R$ is \emph{atomic} if every nonzero nonunit element of~$R$ factors into irreducibles. The class of atomic domains includes important classes of rings such as Noetherian domains and Krull domains \cite[Proposition~2.2]{AAZ90}. We say that $R$ satisfies the \emph{ascending chain condition on principal ideals} (or \emph{ACCP}) if every increasing sequence of principal ideals of $R$ (under inclusion) eventually stabilizes. One can readily check that every integral domain satisfying ACCP is atomic. The converse does not hold in general (see \cite{BC19,GL21,aG74,aZ82}). Here we introduce the following related notion.

\begin{defn}
	An integral domain $R$ is called \emph{hereditarily atomic} if every subring of~$R$ is atomic.
\end{defn}

Although every hereditarily atomic domain is clearly atomic, the converse does not hold (see Example~\ref{ex:atomic overrings vs hereditarily atomic quotient field}). As mentioned in the introduction, the primary purpose of this paper is to study hereditarily atomic domains.
\smallskip

Because most of the monoids we are interested in here are multiplicative monoids of integral domains, within the scope of this paper a \emph{monoid} is a cancellative and commutative semigroup with an identity element denoted by $1$. The quotient monoid $R_{\text{red}} := R^*/R^\times$ is clearly reduced (i.e., its only unit is $1$). We denote the free commutative monoid on $\ii(R_{\text{red}})$ by $\mathsf{Z}(R)$ and call it the \emph{factorization monoid} of $R$. Let $\pi \colon \mathsf{Z}(R) \to R_\text{red}$ be the unique monoid homomorphism fixing the set $\ii(R_{\text{red}})$. For $z = a_1 \cdots a_\ell \in \mathsf{Z}(R)$, where $a_1, \dots, a_\ell \in \ii(R_{\text{red}})$, we call $z$ a \emph{factorization} and say that $|z| := \ell$ is the \emph{length} of $z$. For every $r \in R^*$, set
\[
	\mathsf{Z}(r) := \mathsf{Z}_R(r) := \pi^{-1} (r R^\times) \quad \text{ and } \quad \mathsf{L}(r) := \mathsf{L}_R(r) := \{ |z| \mid z \in \mathsf{Z}(r) \}.
\]
Clearly, $R$ is atomic if and only if $\mathsf{Z}(r)$ is nonempty for every $r \in R^*$. Following \cite{AAZ90}, we say that $R$ is a \emph{bounded factorization domain} (or a \emph{BFD}) provided that $R$ is atomic and $\mathsf{L}(r)$ is finite for every $r \in R^*$. It is not hard to show that every BFD satisfies ACCP (see \cite[Corollary~1.3.3]{GH06}).
\smallskip

Recall that $R$ is a valuation domain if any two (principal) ideals of $R$ are comparable. It is well known that $R$ is a valuation domain if and only if there exist a totally ordered additive group~$G$ and a surjective group homomorphism $v \colon \qf(R)^\times \to G$ such that $R = \{r \in \qf(R) \mid v(x) \ge 0\}$ and $v(x+y) \ge \min\{v(x), v(y)\}$ for all $x,y \in \qf(R)^\times$ with $x+y \neq 0$. In this case, $G$ is unique up to isomorphism and is called the \emph{value group} of $R$. A valuation domain is Noetherian if and only if it is a discrete valuation ring (DVR) \cite[Theorem~5.18]{LM71}. A valuation domain $R$ with value group $G$ is called \emph{discrete} provided that its value group $G$ is discrete, namely, the positive cone of $G$ has a minimum element. Note that DVRs are special cases of discrete valuations.
\smallskip

Dedekind domains and almost Dedekind domains play an important role in some of the results we will establish later. Recall that an almost Dedekind domain is an integral domain whose localizations at maximal ideals are DVRs. Thus, an almost Dedekind domain is a Dedekind domain if and only if it is Noetherian. Unlike Dedekind domains, almost Dedekind domains are not necessarily atomic (see \cite{aG74}). Pr\"ufer domains will also be important in the next section. Recall that a Pr\"ufer domain is an integral domain where every nonzero finitely generated ideal is invertible. It is well known and not hard to verify that every almost Dedekind domain is a Pr\"ufer domain. Unlike almost Dedekind domains, Pr\"ufer domains may have dimension greater than one: the ring consisting of all entire functions on the complex plane is an infinite-dimension Pr\"ufer domain that is not atomic \cite[Section~8.1]{FHP97} and, therefore, non-Noetherian. Pr\"ufer domains are non-Noetherian versions of Dedekind domains: as for almost Dedekind domains, a Pr\"ufer domain is a Dedekind domain if and only if it is Noetherian (see \cite[Corollary~6.7]{LM71}).
\smallskip

Some examples we provide here are based on monoid rings. Let $M$ be an additive monoid. We let~$M^\bullet$ and $U(M)$ denote the set of nonzero elements of $M$ and the group of invertible elements of~$M$, respectively. The \emph{difference group} of $M$ (sometimes called the \emph{Grothendieck group} of $M$) is the unique abelian group $\gp(M)$ up to isomorphism satisfying that any abelian group containing a homomorphic image of $M$ will also contain a homomorphic image of $\gp(M)$. We say that $M$ is \emph{torsion-free} if the group $\gp(M)$ is torsion-free. For an integral domain $R$, the \emph{monoid ring} of $M$ over $R$, denoted by either $R[x;M]$ or $R[M]$, is the ring consisting of all polynomial expressions with exponents in $M$ and coefficients in $R$. When~$M$ is torsion-free, $R[M]$ is an integral domain \cite[Theorem~8.1]{rG84} and the group of units of $R[M]$ is $R[M]^\times = \{uX^m \mid u \in R^\times \ \text{and} \ m \in U(M)\}$ \cite[Theorem~11.1]{rG84}. A comprehensive expository material on the advances in monoid rings (until 1984) is given by Gilmer in~\cite{rG84}.

\bigskip
\section{Integral Domains with Atomic Overrings}
\label{sec:overrings}
\smallskip

It was proved in~\cite{CDM99} that any integral domain can be embedded in an antimatter domain that is not a field. Thus, at the outset, any integral domain is contained in a non-atomic domain. For overrings there are restrictions, as we will see in the first part of this section. Recall that an integral domain~$R$ is Archimedean if for every nonunit $x \in R$ the equality $\bigcap_{n \in \nn} Rx^n = (0)$ holds.

\begin{lem}\label{val}
	For a valuation domain $V$, the following conditions are equivalent.
	\begin{enumerate}
		\item[(a)] $V$ is atomic.
		\smallskip
		
		\item[(b)] $\dim V \le 1$ and $V$ is discrete.
		\smallskip
		
		\item[(c)] $V$ is Noetherian (and so a DVR).
		\smallskip
		
		\item[(d)] $V$ is discrete and Archimedean.
	\end{enumerate}
\end{lem}

\begin{proof}
	(a) $\Rightarrow$ (b): Suppose, by way of contradiction, that there are nonzero prime ideals $P$ and $Q$ such that $P \subsetneq Q$. Let $p$ be a nonzero element of $P$. Since $V$ is atomic, $p$ decomposes into irreducibles, namely, $p = a_1 \cdots a_m$ for some $a_1, \dots, a_m \in \ii(V)$. Because $P$ is prime, $a_j \in P$ for some $j \in \ldb 1,m \rdb$. Take $q \in Q \setminus P$. Since $a_j \nmid_V q$, the fact that $V$ is a valuation domain guarantees that $q \mid_V a_j$. Thus, $q$ and $a_j$ are associates, contradicting that $q \notin P$. Therefore $\dim V \leq 1$. We now note that if $V$ is not discrete, then $V$ has no irreducible elements: this is because the value group of~$V$ has no minimum positive element. 
	\smallskip
	
	(b) $\Leftrightarrow$ (c): This is well known.
	\smallskip
	
	(c) $\Rightarrow$ (a): Every Noetherian domain is atomic (see~\cite[Proposition~2.2]{AAZ90}).
	\smallskip
	
	(c) $\Rightarrow$ (d): This is also clear as every DVR is a discrete valuation and every Noetherian domain is Archimedean by Krull's Intersection Theorem. 
	\smallskip
	
	(d) $\Rightarrow$ (c): It suffices to argue that $\dim V \le 1$. To this end, we suppose that $P$ and $Q$ are prime ideals such that $P \subsetneq Q$. Now take $x \in Q \setminus P$ and $y \in P$. For each $n \in \mathbb{N}$, it is clear that $x^n\in Q \setminus P$, and so $x^n \mid_V y$. Therefore $y \in \bigcap_{n \in \mathbb{N}} Rx^n$, and the fact that $V$ is Archimedean guarantees that $y=0$. Hence $P$ is the zero ideal, and we can conclude that $\dim V \le 1$.
\end{proof}

The following corollary generalizes~\cite[Exercise~2.2.20]{iK74}.

\begin{cor}
	Let $R$ be an integral domain. If every overring of $R$ is atomic, then~$R$ must have valuative dimension at most $1$. In particular, the Krull dimension of $R$ is at most $1$.
\end{cor}

\begin{proof}
	If the valuative dimension of $R$ exceeds 1, then $R$ has a valuation overring of dimension greater than $1$, which cannot be atomic by Proposition~\ref{val}. In general the Krull dimension is bounded above by the valuative dimension (see \cite[Proposition~7.17]{LM71}).
\end{proof}

It is perhaps natural to conjecture at this point that an integral domain has all its overrings atomic if and only if all of the valuations of its quotient field are discrete, but for example, if $R$ is almost Dedekind, but not Dedekind, then it is often not atomic and all of its valuation overrings are discrete and Archimedean.
\smallskip

Besides fields, there are integral domains whose overrings are atomic.

\begin{ex}
	Let $R$ be a one-dimensional Noetherian domain. It is well-known that every overring~$S$ of $R$ is also a one-dimensional Noetherian domain (see \cite[Theorem~93]{iK74}). Since $S$ is Noetherian, it must be atomic.
\end{ex}

Observe that if the quotient field of an integral domain $R$ is hereditarily atomic, then every overring of $R$ is atomic. There are, however, integral domains (that are not fields) having all their overrings atomic whose quotient field is not hereditarily atomic.

\begin{ex} \label{ex:atomic overrings vs hereditarily atomic quotient field}
	Since the ring of polynomials $\qq[x]$ is a Dedekind domain, every overring of $\qq[x]$ is a Dedekind domain and, therefore, atomic. However, its quotient field, $\qq(x)$, is not hereditarily atomic. Indeed, $\zz + x\qq[x]$ is a subring of $\qq(x)$ that is not atomic.
\end{ex}

For our next result, we consider almost Dedekind domains that arise in a specific way, as a directed union of a family of Dedekind domains. This situation mirrors what happens when almost Dedekind domains are constructed from ``large" algebraic extensions of Dedekind domains. Recall that a directed set $\Lambda$ is a partially ordered set satisfying that for all $\alpha, \beta \in \Lambda$, there exists $\theta \in \Lambda$ such that $\alpha \le \theta$ and $\beta \le \theta$. A family of integral domains $\{R_\lambda\}_{\lambda \in \Lambda}$ is called a \emph{directed system} of integral domains provided that $\Lambda$ is a directed set and $R_\alpha \subseteq R_\beta$ if $\alpha \le \beta$. In this case, the integral domain $\bigcup_{\lambda \in \Lambda} R_\lambda$ is the direct limit of the system $\{R_\lambda \}_{\lambda \in \Lambda}$ in the categorical sense.

\begin{defn}
	A directed system $\{R_\lambda\}_{\lambda \in \Lambda}$ of integral domains is called a \emph{directed integral system} of integral domains if, for all $\alpha, \beta \in \Lambda$ with $\alpha \le \beta$, the ring extension $R_\alpha \subseteq R_\beta$ is integral and $[\text{qf}(R_\beta) : \text{qf}(R_\alpha)] < \infty$.
\end{defn}

In general, it is known that the union of a directed system of Pr\"{u}fer domains is a Pr\"{u}fer domain. Here we use atomicity to characterize whether the directed union of a directed integral system of Dedekind domains is a Dedekind domain. Such a characterization will be the main tool that we develop to characterize, in Theorem~\ref{thm:meat}, the fields that are hereditarily atomic.

\begin{theorem}\label{tool}
	Let $D$ be a Pr\"{u}fer domain that is the union of a directed integral system of Dedekind domains. Then all the overrings of $D$ are atomic if and only if $D$ is a Dedekind domain.
\end{theorem}

\begin{proof}
	Since $D$ is a field, the statement of the theorem trivially follows. So we assume that $D$ is not a field. Let $\{ D_\lambda\}_{\lambda \in \Lambda}$ be a directed integral system of Dedekind domains whose union is~$D$. The reverse implication is clear as every overring of a Dedekind domain is again Dedekind~\cite[Corollary~13.2]{rF73} and, therefore, atomic.
	\smallskip

	For the direct implication, assume that every overring of $D$ is atomic. Since $D$ is a Pr\"ufer domain, it is integrally closed. Let us argue that $D$ is integral over $D_\alpha$ for each $\alpha \in \Lambda$. Take $\alpha \in \Lambda$ and $\omega \in D$. As $D$ is the union of the system $\{D_\lambda\}_{\lambda \in \Lambda}$, there exists $\beta \in \Lambda$ such that $\omega \in D_\beta$ and, since the set $\Lambda$ is directed, there is a $\theta \in \Lambda$ such that $D_\alpha, D_\beta \subseteq D_\theta$. As $\omega \in D_\theta$, which is an integral extension of $D_\alpha$, we see that~$\omega$ is integral over $D_\alpha$. Hence $D$ is integral over $D_\alpha$, as desired. Since $D_\alpha$ is a Dedekind domain,~$D$ must be a one-dimensional domain.
	\smallskip
	
	Now we claim that $D$ must be an almost Dedekind domain. To verify this, let $P$ be a prime ideal of~$D$. Since $D$ is a Pr\"ufer domain, $D_P$ is a valuation domain. Since $D_P$ is an overring of $D$, the former must be atomic by hypothesis. Thus, Lemma~\ref{val} ensures that $D_P$ is a DVR. As a consequence, $D$ is an almost Dedekind domain.
	\smallskip
	
	Finally, let us assume towards a contradiction that $D$ is not a Dedekind domain. Since $D$ is almost Dedekind but not Dedekind, \cite[Theorem~3]{rG64} guarantees the existence of a nonzero proper ideal of $D$ that is contained in infinitely many maximal ideals. So the fact that $D$ is atomic allows us to take $a \in \ii(D)$ such that $a$ is contained in infinitely many prime ideals of $D$. Take $\alpha \in \Lambda$ such that $a \in D_\alpha$ and record the prime ideal factorization of $a D_\alpha$ in $D_\alpha$, namely,
	\[
		a D_\alpha = \mathfrak{P}_1^{e_1}\mathfrak{P}_2^{e_2} \cdots \mathfrak{P}_k^{e_k}
	\]
	for distinct prime ideals $\mathfrak{P}_1, \dots, \mathfrak{P}_k$ of $D_\alpha$ and $e_1, \dots, e_k \in \nn$. As every prime ideal of $D$ containing~$a$ lies over one of the ideals $\mathfrak{P}_1, \dots, \mathfrak{P}_k$, at least one of such ideals, namely $\mathfrak{P}$, must be contained in infinitely many prime ideals of $D$. Write $S := D_\alpha \setminus \mathfrak{P}$. 
	\smallskip
	
	For each $\beta \in \Lambda$ with $\beta \ge \alpha$, it is clear that $S$ is a multiplicative subset of $D_\beta$, and so we can consider the integral domain $R_\beta := S^{-1}D_\beta$. Observe that $\{R_\beta\}_{\beta \ge \alpha}$ form a directed integral system whose union is the localization $R: = S^{-1}D$. Since $R$ is a localization of the almost Dedekind domain~$D$, it is an almost Dedekind domain by \cite[Theorem~4(a)]{rG64}. Because there are infinitely many prime ideals of $D$ lying over $\mathfrak{P}$, the localization correspondence theorem guarantees that there are infinitely many prime ideals of $R$ lying over $S^{-1}\mathfrak{P}$. Therefore the nonzero ideal $Ra$ is contained in infinitely many prime ideals of $R$, and so $R$ is not a Dedekind domain.
	\smallskip
	
	Fix $\beta \in \Lambda$ with $\beta \ge \alpha$. Since the extension $R_\alpha \subseteq R_\beta$ is integral, $R_\beta$ is the integral closure of~$R_\alpha$ in $\qf(R_\beta)$. Now as $[\qf(R_\beta) : \qf(R_\alpha)] < \infty$, it follows, as a consequence of Krull-Akizuki Theorem (see~\cite[Theorem~11.7]{hM86}), that $R_\beta$ contains only finitely many prime ideals lying over $S^{-1}\mathfrak{P}$. Since $S^{-1}\mathfrak{P}$ is the only prime ideal of $R_\alpha$, we see that $R_\beta$ is a Dedekind domain with only finitely many prime ideals and, therefore, it is a PID. Thus, $R_\beta$ is a PID provided that $\beta \ge \alpha$. Because $R$ is an atomic domain (it is an overring of $D$), we can factor $a$ in $R$ as follows:
	\[
		a = a_1 a_2 \cdots a_m,
	\]
	where $a_1, \dots, a_m \in \ii(R)$. Since $R$ is the union of the directed integral system $\{R_\beta\}_{\beta \ge \alpha}$, there is a $\theta \in \Lambda$ with $\theta \ge \alpha$ such that $a_1, \dots, a_m \in R_\theta$. Since the extension $D_\theta \subseteq D$ is integral, so is $R_\theta \subseteq R$. Then it follows that $R^\times \cap R_\theta = R_\theta^\times$, and so $\ii(R) \cap R_\theta \subseteq \ii(R_\theta)$. Therefore $a_1, \dots, a_m \in \ii(R_\theta)$, and the fact that $R_\theta$ is a PID guarantees that $a_1, \dots, a_n$ are primes in $R_\theta$. As $a$ is contained in infinitely many prime ideals of $R$, we can assume, without loss of generality, that $a_1$ is not prime in $R$. So there exist $r_1, r_2, r_3 \in R$ such that $r_1 r_2 = r_3 a_1$ with $a_1$ not dividing either $r_1$ or $r_2$ in $R$. In particular, $r_1, r_2 \notin R^\times$. Choose $\omega \in \Lambda$ with $\omega \ge \alpha$ such that $r_1, r_2, r_3, a_1 \in R_\omega$. Since the extension $R_\omega \subseteq R$ is integral and $a_1$ is irreducible in $R$, it must be irreducible, and hence prime, in $R_\omega$. Then $a_1$ must divide $r_i$ in $R_\omega$ for some $i \in \{1,2\}$. As a result, $a_1$ would divide $r_i$ in $R$, which is a contradiction. Hence we conclude that $D$ is a Dedekind domain.
\end{proof}

Unlike Pr\"ufer domains, the directed union of Dedekind domains may not be a Dedekind domain.

\begin{ex} \label{ex:direct union of DDs that is not a DD}
	Let $\Lambda$ be the set consisting of all finite subsets of $\overline{\zz}_\cc$. It is clear that $\Lambda$ is a directed set under inclusion. In addition, for each $\lambda \in \Lambda$, the ring $R_\lambda : = \zz[\lambda]$ is a finite integral extension of $\zz$ and, therefore, it is a Dedekind domain. Observe, on the other hand, that the union of the directed system $\{R_\lambda\}_{\lambda \in \Lambda}$ is $\overline{\zz}_\cc$, which is not even Noetherian (for instance, the ideal $(\sqrt[n]{2} \mid n \in \nn)$ is not finitely generated).
\end{ex}

\bigskip
\section{Hereditarily Atomic Fields}
\label{sec:fields}
\smallskip

As we have seen in Example~\ref{ex:atomic overrings vs hereditarily atomic quotient field}, having a hereditarily atomic quotient field is a stronger condition than having all overrings atomic. Our primary purpose in this section is to classify the fields that are hereditarily atomic. In doing so, the results we have established in Section~\ref{sec:overrings} will be useful.

To begin with, we observe that if a field $F$ is a finite-dimensional extension of its prime field, then every subring of $F$ is Noetherian by \cite[Theorem]{rG70}. Since every Noetherian domain is a BFD by \cite[Proposition~2.2]{AAZ90}, we immediately obtain the following sufficient condition for a field to be hereditarily atomic.

\begin{prop}
	Let $F$ be a field. If $F$ is a finite-dimensional extension of its prime field, then $F$ is hereditarily atomic.
\end{prop}

As we will see later, the previous sufficient condition is not necessary. Indeed, characterizing the entirety of the family of hereditarily atomic fields is more delicate. For now let us collect two necessary conditions for a field to be hereditarily atomic.

\begin{prop} \label{prop:necessary conditions for HAFs}
	Let $F$ be a field. If $F$ is hereditarily atomic, then the following statements hold.
	\begin{enumerate}
		
		\item If $\emph{\ch}(F) = 0$, then $F$ is algebraic over $\mathbb{Q}$.
		\smallskip
		
		\item If $\emph{\ch}(F) = p \in \pp$, then the transcendence degree of $\ff_p \subseteq F$ is at most $1$.
		\smallskip
		
		\item Every valuation of $F$ is both discrete and Archimedean.
	\end{enumerate}
\end{prop}

\begin{proof}
	(1) Suppose first that $\ch(F) = 0$. Assume, by way of contradiction, that $F$ is not algebraic over $\qq$, and consider the subring $R = \zz + t \qq[t]$ of $F$, where $t \in F$ is a transcendental element over~$\qq$. It is clear that $R$ is isomorphic to the subring $\zz + x \qq[x]$ of the ring of polynomials $\qq[x]$. Since~$\zz$ is not a field, it follows from \cite[Corollary~1.4]{AeA99} that $R$ is not atomic, which is a contradiction.
	\smallskip
	
	(2) Suppose now that $\ch(F) = p \in \pp$. Similar to part~(1), if $t_1, t_2 \in F$ were algebraically independent over $\ff_p$, then the subring $\ff_p[t_1] + t_2 \ff_p(t_1)[t_2]$ of $F$ would be isomorphic to the non-atomic subring $\ff_p[t_1] + x \ff_p(t_1)[x]$ of the polynomial ring $\ff_p(t_1)[x]$.
	\smallskip
	
	(3) For this, it suffices to note that if $F$ has a valuation that is not discrete or a valuation that is not Archimedean, then its corresponding valuation domain will not be atomic by Lemma~\ref{val}.	
\end{proof}

In light of Proposition~\ref{prop:necessary conditions for HAFs}, in order to classify the fields that are hereditarily atomic, it suffices to restrict our attention to algebraic extensions of $\qq$ and field extensions of $\ff_p$ of transcendence degree at most $1$. However, not all such field extensions are hereditarily atomic. The following example sheds some light upon this observation.

\begin{ex}
	Take $p \in \pp$, and consider the additive submonoid $M = \zz[1/p]_{\ge 0}$ of $\qq_{\ge 0}$, where $\zz[1/p]$ is the localization of $\zz$ at the multiplicative set $\{p^n \mid n \in \nn_0\}$. Now consider the monoid ring $\ff_p[M]$ of $M$ over $\ff_p$. Since $\ff_p$ is a perfect field and $M = pM$, every polynomial expression in $\ff_p[M]$ is a $p$-th power in $\ff_p[M]$. As a consequence, $\ff_p[M]$ is an antimatter domain. As $\ff_p[M]$ is not a field, it cannot be atomic. Finally, observe that $\ff_p[M]$ is a subring of the algebraic extension $\ff_p(x^m \mid m \in M)$ of the field $\ff_p(x)$.
\end{ex}

We are in a position now to characterize the fields that are hereditarily atomic.

\begin{theorem}\label{thm:meat}
	Let $F$ be a field. 
	\begin{enumerate}
		\item If $\emph{char}(F) = 0$, then $F$ is hereditarily atomic if and only if $F$ is an algebraic extension of $\mathbb{Q}$ such that $\overline{\mathbb{Z}}_{F}$ is a Dedekind domain.
		\smallskip
		
		\item If $\emph{char}(F) = p \in \pp$, then $F$ is hereditarily atomic if and only if the transcendental degree of~$F$ over $\ff_p$ is at most~$1$ and $\overline{\ff_p[x]}_F$ is a Dedekind domain for every $x \in F$.
	\end{enumerate}
\end{theorem}

\begin{proof}
	(1) For the direct implications, suppose that $F$ is hereditarily atomic. It follows from Proposition~\ref{prop:necessary conditions for HAFs} that $F$ is an algebraic extension of $\qq$. Since $\zz$ is a Pr\"ufer domain with quotient field $\qq$, it follows that $\overline{\zz}_F$ is a (one-dimensional) Pr\"ufer domain \cite[Theorem~101]{iK74}. Note now that the set of all finite sub-extensions of $\overline{\zz}_F$ is a directed system of Dedekind domains. This directed system is clearly a directed integral system with directed union $\overline{\zz}_F$. As $F$ is hereditarily atomic, every overring of $\overline{\zz}_F$ is atomic and, therefore, Theorem~\ref{tool} guarantees that $\overline{\zz}_F$ is a Dedekind domain.
	
	Conversely, suppose that $F$ is an algebraic extension of $\qq$ such that $\overline{\zz}_F$ is a Dedekind domain. Let~$S$ be a subring of $F$. As $\overline{\zz}_F \subseteq \overline{S}_F$ holds, the fact that $F$ is algebraic over both $\qq$ and $\qf(S)$ guarantees that $\qf(\overline{\zz}_F) = F = \qf(\overline{S}_F)$. Therefore $\overline{S}_F$ is an overring of $\overline{\zz}_F$. Since $\overline{\zz}_F$ is a Dedekind domain, so is $\overline{S}_F$ \cite[Theorem~6.21]{LM71}. Thus, $\overline{S}_F$ is Noetherian and, in particular, it must satisfy ACCP. On the other hand, $\overline{S}_F^\times \cap S = S^\times$ because $S \subseteq \overline{S}_F$ is an integral extension. From this, one infers that~$S$ also satisfies ACCP. As a consequence, $S$ is atomic. Hence $F$ is hereditarily atomic.
	\smallskip
	
	(2) Suppose first that $F$ is hereditarily atomic. It follows from Proposition~\ref{prop:necessary conditions for HAFs} that the transcendence degree of $F$ over $\ff_p$ is at most~$1$. Assume that $F$ is not algebraic over $\ff_p$, and fix a transcendental $x \in F$ over $\ff_p$. Then $F$ is an algebraic extension of $\ff_p(x)$. Since $\ff_p[x]$ is a Pr\"ufer domain and $\overline{\ff_p[x]}_F$ is one-dimensional, we can mimic the argument in the first paragraph of this proof to conclude that $\overline{\ff_p[x]}_F$ is a Dedekind domain.
	
	For the reverse implications, we first observe that if $F$ is an algebraic extension of $\ff_p$ for some $p \in \pp$, then every subring of $F$ is a field, whence $F$ is hereditarily atomic. We suppose, therefore, that the transcendental degree of $F$ over $\ff_p$ is~$1$. Let $S$ be a subring of $F$. If every element of $S$ is algebraic over $\ff_p$, then $S$ is a field, and so atomic. Otherwise, let $x \in S$ be a transcendental element over $\ff_p$. Then $\ff_p[x]$ is a subring of $S$. Since $F$ is an algebraic extension of $\ff_p(x)$ and $\overline{\ff_p[x]}_F$ is a Dedekind domain, we can argue that $S$ is hereditarily atomic by simply following the lines of the second paragraph of this proof. Thus, $F$ is hereditarily atomic.
\end{proof}

We record the following corollary (actually a porism to Theorem~\ref{thm:meat}).

\begin{cor}\label{por}
	The field $F$ is hereditarily atomic if and only if each subring of $F$ is a BFD.
\end{cor}

\begin{cor} \label{cor:F(x) is HA when $F$ is an algebraic extension of F_p}
	For $p \in \pp$, let $F$ be an algebraic extension of $\ff_p$. Then the field $F(x)$ is hereditarily atomic.
\end{cor}

\begin{proof}
	It suffices to show that $F(x)$ is hereditarily atomic assuming that $F$ is the algebraic closure of~$\ff_p$. It is clear that $\overline{\ff_p[x]}_{F(x)} \subseteq \overline{F[x]}_{F(x)} = F[x]$ because $F[x]$ is integrally closed. Conversely, every element of $F[x]$ is integral over $\ff_p[x]$, and so $F[x] \subseteq \overline{\ff_p[x]}_{F(x)}$. Thus, $\overline{\ff_p[x]}_{F(x)} = F[x]$ is a Dedekind domain, and it follows from part~(2) of Theorem~\ref{thm:meat} that $F(x)$ is hereditarily atomic. 
\end{proof}

We quickly note that it is possible to have a Dedekind domain $D$ with quotient field $F$ such that~$F$ is algebraic, but not finite, over $\mathbb{Q}$ (resp. $\ff_p(x)$). First, we record a key result needed in the construction. This result is \cite[Theorem 42.5]{rG68}.

\begin{theorem}\label{gilmer}
	Let $R$ be a Dedekind domain with quotient field $K$, and let 
	\[
		\{P_i\}_{i=1}^r, \{Q_i\}_{i=1}^s, \{U_i\}_{i=1}^t
	\]
	be three collections of distinct maximal ideals of $R$ (with $r \ge 1$), each with finite residue field. Then there is a simple quadratic field extension $K(t)$ of~$K$ (with ring of integers $\overline{R}_{K(t)}$) such that~$t$ integral over~$R$ and separable over~$K$ and each~$P_i$ is inertial in $\overline{R}_{K(t)}$, each $Q_i$ ramifies in $\overline{R}_{K(t)}$, and each $U_i$ splits in $\overline{R}_{K(t)}$.
\end{theorem}

\begin{ex}
	Let $R = R_0$ be a Dedekind domain, and suppose that $R$ has (only) countably many nonzero prime ideals $(\mathfrak{P}_n)_{n \in \nn}$, each of them with finite residue field. We can apply Theorem~\ref{gilmer} to build a quadratic field extension $K_1$ of $\text{qf}(R)$ such that $\mathfrak{P}_1$ is inertial in $R_1 := \overline{R}_{K_1}$ (and leave the set to ramify and the set to split both as empty). Since $R_1$ is Dedekind, we can apply Theorem~\ref{gilmer} to obtain a quadratic field extension $K_2$ of $\text{qf}(R_1)$ such that the primes of $R_1$ lying over $\mathfrak{P}_2$ and $\mathfrak{P}_1$ (which is only $\mathfrak{P}_1R_1$ itself) are inertial in $R_2 := \overline{R_1}_{K_2}$. In the $n^{\text {th}}$ step we construct a quadratic field extension $K_n$ of $\text{qf}(R_{n-1})$ such that the primes of $R_{n-1}$ lying over $\mathfrak{P}_1,\mathfrak{P}_2, \dots, \mathfrak{P}_n$ are inertial in $R_n := \overline{R_{n-1}}_{K_n}$. Then we can consider $D = \bigcup_{n \in \nn} R_n$. It follows from \cite[Corollary 42.2]{rG68} that $D$ is an almost Dedekind domain. Since any $\mathfrak{P}_n$ can be contained in only finitely many (in fact no more than $2^{n-1}$) prime ideals of $D$, we obtain that~$D$ is a Dedekind domain. However, $\text{qf}(D)$ is an infinite field extension of~$\text{qf}(R)$.
\end{ex}

\bigskip
\section{Polynomial-Like Rings}
\label{sec:polynomials}
\smallskip

The main purpose of this section is to study hereditary atomicity in rings of polynomials and rings of Laurent polynomials. To begin with, we will determine which are the fields whose polynomial rings are hereditarily atomic. Every hereditarily atomic domain we have seen so far lives inside a hereditarily atomic field. We will see that there are rings of polynomials that are hereditarily atomic domains but whose quotient fields are not hereditarily atomic and, therefore, they cannot be embedded into any hereditarily atomic field.
\smallskip

First, we note here that the distinction between atomic domains and domains satisfying ACCP is notoriously subtle, and just a few classes of atomic domains without ACCP have been found so far (see, for instance, \cite{BC19}, \cite{aG74}, and~\cite{aZ82}). The following proposition, which we will use later, is not hard to verify.

\begin{prop} \cite[Proposition~2.1]{aG74} \label{prop:ACCP subrings}
	Let $R$ be an integral domain satisfying ACCP. If a subring~$S$ of $R$ satisfies $S^\times = R^\times \cap S$, then $S$ also satisfies ACCP.
\end{prop}

%

Using Proposition~\ref{prop:ACCP subrings}, one can produce examples of hereditarily atomic domains whose quotient fields are not hereditarily atomic.

\begin{ex} \label{ex:rings of polynomials over Z are HAD}
	For $n \in \nn$, consider the polynomial ring $R = \zz[x_1, \dots, x_n]$. Since $R^\times = \{\pm 1\}$, for each subring $S$ of $R$ we see that $S^\times = \{\pm 1\} = R^\times \cap S$. It follows now from Proposition~\ref{prop:ACCP subrings} that~$S$ satisfies ACCP and is, therefore, atomic. Thus, $R$ is hereditarily atomic. However, the quotient field of $R$ is not hereditarily atomic by Theorem~\ref{thm:meat}.
\end{ex}

\begin{ex} \label{ex:rings of polynomials over Z_2 that are HAD}
	For $n \in \nn$, the only unit of the polynomial ring $R = \ff_2[x_1, \dots, x_n]$ is $1$. Now we can proceed as in Example~\ref{ex:rings of polynomials over Z are HAD} to obtain that $R$ is hereditarily atomic even though Theorem~\ref{thm:meat} guarantees that the quotient field of $R$ is not hereditarily atomic when $n \ge 2$. 
\end{ex}

It is well known that if an integral domain $R$ satisfies ACCP (resp., is a BFD), then the ring of polynomials $R[x]$ over $R$ satisfies ACCP (resp., is a BFD). This is not the case for being atomic as proved by M. Roitman~\cite{mR93}. Similarly, the ring $R[x]$ may not be hereditarily atomic even when $R$ is hereditarily atomic. Indeed, the field~$\qq$ is hereditarily atomic by Theorem~\ref{thm:meat}, but the ring $\qq[x]$ contains the subring $\zz + x \qq[x]$, which is not atomic. We proceed to characterize the fields yielding hereditarily atomic polynomial rings.

\begin{prop} \label{prop:polynomial rings that are HAD}
	Let $F$ be a field. Then the following statements are equivalent.
	\begin{enumerate}
		\item[(a)] $F[x_1, \dots, x_n]$ is hereditarily atomic for any distinct indeterminates $x_1, \dots, x_n$.
		\smallskip
		
		\item[(b)] $F[x]$ is hereditarily atomic.
		\smallskip
		
		\item[(c)] $F$ is an algebraic extension of $\ff_p$ for some $p \in \pp$.
	\end{enumerate}
\end{prop}

\begin{proof}
	(a) $\Rightarrow$ (b): This is clear.
	\smallskip
	
	(b) $\Rightarrow$ (c): Suppose that $F[x]$ is hereditarily atomic. As we have observed, $\ch(F)$ cannot be zero, as otherwise $F[x]$ would contain a copy of the non-atomic domain $\zz + x\qq[x]$. Let $p$ be the characteristic of $F$. Since every subring of $F$ is atomic, it follows from part~(b) of Theorem~\ref{thm:meat} that $F$ is an algebraic extension of either $\ff_p$ or $\ff_p(y)$. Note, however, that~$F$ cannot be an extension of $\ff_p(y)$ because, in such a case, $\ff_p[y] + x\ff_p(y)[x]$ would be a non-atomic subring of $F[x]$ by \cite[Corollary~1.4]{AeA99}. Hence $F$ must be an algebraic extension of $\ff_p$.
	\smallskip
	
	(c) $\Rightarrow$ (a): Suppose now that $F$ is an algebraic extension of the finite field $\ff_p$ for some $p \in \pp$, and set $R = F[x_1, \dots, x_n]$. Let~$S$ be a subring of $R$. If $S$ is a subring of $F$, then $S$ must be atomic by Theorem~\ref{thm:meat} (indeed, in this case, $S$ is a field). Assume, therefore, that $S$ is not a subring of $F$, and take $a \in S \setminus F$ satisfying that for all $g,h \in R$ with $a = gh$, either $g \in F$ or $h \in F$. Write $a = g h$ for some $g, h \in S$ and suppose, without loss of generality, that $g \in F$. Since $F \cap S$ is a subring of $F$, it must be a field, and so $g \in F \cap S$ implies that $g \in S^\times$. Thus, $a \in \ii(S)$. Finally, if $f \in S$, then we can write $f = a_1 \cdots a_m$ for $a_1, \dots, a_m \in S \setminus F$ taking $m \in \nn$ as large as it can be. In this case, $a_1, \dots, a_m$ must be irreducibles in~$S$, and so $a_1 \cdots a_m$ is a factorization of~$f$ in~$S$. Hence $S$ is atomic, and we can conclude that $R$ is hereditarily atomic.
\end{proof}

The following corollary is an immediate consequence of Theorem~\ref{thm:meat} and Proposition~\ref{prop:polynomial rings that are HAD}.

\begin{cor} \label{cor:polynomial rings of several variables over F_p are HAD}
	Let $F$ be an algebraic extension of $\ff_p$ for some $p \in \pp$. For $n \ge 2$, the ring of polynomials $F[x_1, \dots, x_n]$ is a hereditarily atomic domain whose quotient field is not hereditarily atomic.
\end{cor}

Although Noetherian and Krull domains are BFDs, they are not necessarily hereditarily atomic (see Corollary~\ref{cor: Noetherian/Krull domains not HAD}). However, as a consequence of Proposition~\ref{prop:polynomial rings that are HAD}, we can identify hereditarily atomic Noetherian and Krull domains of any dimension.

\begin{cor}  \label{cor: Noetherian/Krull domains HAD}
	For every $n \in \nn$ there exists a Noetherian/Krull domain of dimension $n$ that is hereditarily atomic.
\end{cor}

\begin{proof}
	For each $n \in \nn$, the $n$-dimensional ring $\ff_p[x_1, \dots, x_k]$ is both a Noetherian domain and a Krull domain, and it is also hereditarily atomic by Proposition~\ref{prop:polynomial rings that are HAD}.
\end{proof}
\smallskip

For every $n \in \nn$, it follows from Corollary~\ref{cor:polynomial rings of several variables over F_p are HAD} that the integral domain $\ff_p[x_1, \dots, x_n]$, which is isomorphic to the monoid ring of the rank-$n$ free commutative monoid $\nn_0^n$ over the field $\ff_p$, is hereditarily atomic. The next example exhibits another class of monoids whose monoid rings over~$\ff_p$ are hereditarily atomic.

First, let us extend some terminology from the setting of integral domains to that of monoids. The notions of irreducibility and atomicity carry over to any monoid $M$ as follows: $a \in M \setminus U(M)$ is \emph{irreducible} if $a = xy$ for some $x,y \in M$ implies that either $x \in U(M)$ or $U(M)$, and $M$ is \emph{atomic} if every non-invertible element factors into irreducibles. Moreover, we can extend in the obvious way the notion of a factorization to atomic monoids, and so we can define a \emph{bounded factorization monoid} (\emph{BFM}) mimicking the definition of a BFD. Finally, a subset $I$ of a monoid $M$ is called an \emph{ideal} if $I + M \subseteq I$, and so we can consider monoids satisfying ACCP. As for integral domains, every BFM satisfies ACCP \cite[Corollary~1.3.3]{GH06} and every monoid satisfying ACCP is atomic \cite[Proposition~1.1.4]{GH06}.

\begin{ex}
	Fix $p \in \pp$. Let $M$ be an additive submonoid of $\rr_{\ge 0}$ such that $0$ is not a limit point of $M^\bullet$. Since~$0$ is not a limit point of $M^\bullet$, it follows from~\cite[Proposition~4.5]{fG19} that $M$ is a BFM, and so the monoid ring $\ff_p[M]$ is a BFD by \cite[Theorem~13]{AJ15}. In particular, $\ff_p[M]$ satisfies ACCP. Now suppose that $S$ is a subring of $\ff_p[M]$. As the only invertible element of $M$ is $0$, we see that $\ff_p[M]^\times = \ff_p^\times$, and so the equality $S^\times = \ff_p^\times = \ff_p[M]^\times \cap S$ holds. Then it follows from Proposition~\ref{prop:ACCP subrings} that~$S$ satisfies ACCP and is, therefore, atomic. As a result, $\ff_p[M]$ is a hereditarily atomic domain. Several aspects of the atomicity of $F[M]$ for additive submonoids $M$ of $\qq_{\ge 0}$ (and any field $F$) were recently studied by the second author in~\cite{fG21}.
\end{ex}
\medskip

Given a polynomial ring $F[x]$ over a field, subrings of the form $D + xF[x]$, where $D$ is a subring of~$F$, are special cases of the well-investigated $D+M$ construction, which was first studied by Gilmer \cite[Appendix II]{rG68} in the context of valuation domains, and later by J. Brewer and E. A. Rutter~\cite{BR76} for arbitrary integral domains. For an integral domain $T$, let~$K$ and~$M$ be a subfield of $T$ and a nonzero maximal ideal of $T$, respectively, such that $T = K + M$. Now let $D$ be a subring of $K$, and set $R = D + M$. It is known that $R$ is atomic (resp., satisfies ACCP) if and only if $T$ is atomic (resp., satisfies ACCP) and $D$ is a field \cite[Proposition~1.2]{AAZ90}. If $T$ has characteristic zero, then $R$ cannot be hereditarily atomic.

\begin{prop}
	Let $T$ be an integral domain, and let $K$ and $M$ be a subfield of $T$ and a nonzero maximal ideal of $T$, respectively, such that $T = K + M$. For a subring $D$ of~$K$, set $R = D + M$. If $T$ has characteristic zero, then $R$ is not hereditarily atomic.
	%
\end{prop}

\begin{proof}
	Suppose that $T$ (and so $R$) has characteristic zero. If $D$ is not a field, then it follows from \cite[Proposition~1.2]{AAZ90} that $R$ is not even atomic. Therefore we assume that $D$ is a field extension of~$\qq$. We claim that every nonzero element of $M$ is transcendental over $D$. Indeed, if for some nonzero $m \in M$ there existed a polynomial $f(x) = \sum_{j=0}^n c_j x^j \in D[x]$ with $c_0 \neq 0$ and $f(m) = 0$, then the equality $0 = c_0 + m \sum_{j=1}^n c_j m^{j-1}$ would contradict that $D+M$ is a direct sum of abelian groups. As a result,~$R$ contains an isomorphic copy of $\qq[x]$ as a subring, namely, $\qq[m]$ for any nonzero $m \in M$. Since $\qq[x]$ is not hereditarily atomic (it contains $\zz + x \qq[x]$ as a subring), neither $R$ nor $T$ can be hereditarily atomic.
\end{proof}

If~$T$ is hereditarily atomic (which can only happen if $\text{char}(T)$ is prime), then $R$ is clearly hereditarily atomic (even if $D$ is not a field). On the other hand, if $R$ is hereditarily atomic, then it is atomic and so it follows from \cite[Proposition~1.2]{AAZ90} that $T$ is atomic and $D$ is a field. However, unlike for the property of being atomic, the fact that $R$ is hereditarily atomic does not suffice to guarantee that $T$ is hereditarily atomic. The following example sheds some light upon this last observation.

\begin{ex}
	Fix $p \in \pp$, and let $x$ be an indeterminate over $\ff_p$. Let $T$ denote the polynomial ring $\ff_p(x)[y]$ in the indeterminate $y$ over the field $\ff_p(x)$, and let $R := \ff_p + y \ff_p(x)[y]$ be the subring of~$T$ of the form $D+M$, where $D = \ff_p$ and $M = y \ff_p(x)[y]$. Since $T$ satisfies ACCP (indeed, $T$ is a UFD), \cite[Proposition~1.2]{AAZ90} ensures that~$R$ also satisfies ACCP. It follows from \cite[Lemma~4.17]{AG21} that $R^\times = T^\times \cap R = \ff_p^\times$. Now suppose that~$S$ is a subring of~$R$, and observe that $R^\times \cap S =  \ff_p^\times = S^\times$. Then Proposition~\ref{prop:ACCP subrings} guarantees that $S$ satisfies ACCP and is, therefore, atomic. Hence $R$ is hereditarily atomic. However, it follows from Proposition~\ref{prop:polynomial rings that are HAD} that $T$ is not hereditarily atomic.
\end{ex}
\smallskip

We have just identified in Proposition~\ref{prop:polynomial rings that are HAD} a class of polynomial rings that are hereditarily atomic. In the next proposition we fully characterize the rings of Laurent polynomials that are hereditarily atomic.

\begin{theorem} \label{thm:Laurent polynomial rings}
	Let $R$ be an integral domain. Then $R[x,x^{-1}]$ is hereditarily atomic if and only if~$R$ is an algebraic extension of $\ff_p$ for some $p \in \pp$.
\end{theorem}

\begin{proof}
	For the reverse implication, suppose that $R$ is an algebraic extension of $\ff_p$ for some $p \in \pp$. In particular, we can assume that $R$ is a subfield of the algebraic closure $K$ of $\ff_p$. Then $R[x,x^{-1}]$ is a subring of $K(x)$. Since $K(x)$ is hereditarily atomic by Corollary~\ref{cor:F(x) is HA when $F$ is an algebraic extension of F_p}, we obtain that $R[x,x^{-1}]$ is also hereditarily atomic.
	\smallskip
	
	For the direct implication, suppose that $R[x,x^{-1}]$ is hereditarily atomic. We will argue first that~$R$ cannot have characteristic zero. Suppose, towards a contradiction, that $\text{char}(R) = 0$. Then $R[x, x^{-1}]$ contains $\zz[x, x^{-1}]$ as a subring. Now consider the subring
	\[
		T := \zz[x^n, 2/x^n \mid n \in \nn]
	\]
	of $\zz[x,x^{-1}]$. One can readily verify that $x^{-1} \notin T$. This, along with the inclusion $T^\times \subseteq \{\pm x^n \mid n \in \zz\}$, implies that $T^\times = \{\pm 1\}$. As a consequence, $2/x^j = x(2/x^{j+1})$ ensures that $2/x^j \notin \ii(T)$ for any $j \in \nn_0$. Since~$T$ is a subring of $R[x,x^{-1}]$, it must be atomic. Therefore we can write $2 = a_1(x) \cdots a_n(x)$ for some $a_1(x), \dots, a_n(x) \in \ii(T)$, where $n \ge 2$ because $2 \notin T^\times \cup \ii(T)$. Since the monoid
	\[
		M := \{cx^n \mid c \in \zz \setminus \{0\} \text{ and } n \in \zz\}
	\]
	is a divisor-closed submonoid of the multiplicative monoid $\zz[x,x^{-1}] \setminus \{0\}$, it follows that $a_j(x) \in M$ for every $j \in \ldb 1,n \rdb$. After relabeling and taking associates, we can assume that $a_1(x) = 2x^{s_1}$ and $a_j(x) = x^{s_j}$ for every $j \in \ldb 2,n \rdb$, where $s_1, \dots, s_n \in \zz$ satisfy $s_1 + \dots + s_n = 0$. The fact that $a_1(x) \in \ii(T)$ guarantees that $s_1 \ge 1$. Thus, there must be a $k \in \ldb 2,n \rdb$ such that $s_k < 0$. However, in this case $a_k(x) = x^{s_k}$ would be a unit of $T$, which is a contradiction.
	\smallskip
	
	Hence $\text{char}(R) = p$ for some $p \in \pp$, and so $R$ contains $\ff_p$ as its prime subfield. We will argue that~$R$ is an algebraic extension of $\ff_p$. Assume, by way of contradiction, that there exists $w \in R$ that is transcendental over $\ff_p$. Let us prove that the subring
	\[
		S := \ff_p[x, w/x^n \mid n \in \nn_0]
	\]
	of $R[x,x^{-1}]$ is not atomic, which will yield the desired contradiction.
	
	If $x$ were a unit of $S$, then $x^{-1} = \sum_{i=0}^k g_i(x) w^i$ for some $g_0 \in \ff_p[x]$ and $g_1, \dots, g_k \in \ff_p[x,x^{-1}]$, and so the fact that $w$ is transcendental over $\ff_p(x)$ (as an element of the extension $\qf(R)(x)$) would imply that $x g_0(x) = 1$. Thus, $x \notin S^\times$. Similarly, if $w/x^n$ were a unit of $S$ for some $n \in \nn_0$, then $x^n/w = \sum_{i=0}^k g_i(x) w^i$ for some $g_0, \dots, g_k \in \ff_p[x,x^{-1}]$ and, after clearing denominators, we would obtain $F(x,w) = 0$ for some nonzero $F$ in the polynomial ring $\ff_p[X,W]$, contradicting that $\{x,w\}$ is algebraically independent over $\ff_p$. Thus, $w/x^n \notin S^\times$ for any $n \in \nn_0$.
	
	We proceed to argue that $w$ does not factor into irreducibles in $S$. Observe first that for every $m \in \nn_0$ the element $w/x^m$ is not irreducible in $S$ because $w/x^m = x(w/x^{m+1})$ and, as we have already checked, $x, w/x^{m+1} \notin S^\times$. In particular, $w$ is not irreducible in $S$. Now write
	\begin{equation} \label{eq:product}
		w = \prod_{i=1}^n f_i\Big(x, w, \frac wx, \dots, \frac w{x^k} \Big) 
	\end{equation}
	in~$S$ for some $k \in \nn_0$, $n \in \nn_{\ge 2}$, and $f_1, \dots, f_n \in \ff_p[X, W_0, \dots, W_k]$ such that $f_i(x,w, w/x, \dots, w/x^k)$ is not a unit of $S$ for any $i \in \ldb 1,n \rdb$. Observe that for every $i \in \ldb 1,n \rdb$, there is a unique $\ell_i \in \zz$ such that $g_i := x^{-\ell_i} f_i \in \ff_p[x,w]$ and $g_i(0,w) \neq 0$. Set $\ell = \ell_1 + \dots + \ell_n$. It follows from~\eqref{eq:product} that $g_1(x,w) \cdots g_n(x,w) = x^{-\ell} w$, which implies that $\ell = 0$. Therefore $w = g_1(x,w) \cdots g_n(x,w)$ in $\ff_p[x,w]$. Since $\ff_p[x,w]$ is a UFD and $w$ is irreducible in $\ff_p[w,x]$, we can assume, after a possible relabeling of subindices, that $g_1 = \alpha_1 w$ for some $\alpha_1 \in \ff_p^\times$ and $g_i = \alpha_i \in \ff_p^\times$ for every $i > 1$. Thus, $f_1 = \alpha_1 x^{\ell_1}w$ and $f_i = \alpha_i x^{\ell_i}$ for every $i > 1$. Now we observe that $f_1$ is not irreducible in $S$ because $f_1 = \alpha_1 x^N (w/x^{N-\ell_1})$, where $N := |\ell_1| + 1$, and neither $\alpha_1 x^N$ nor $w/x^{N-\ell_1}$ is a unit of $S$. As a consequence, $w$ does not factor into irreducibles in $S$, and so $S$ is a subring of $R[x,x^{-1}]$ that is not atomic, contradicting that $R[x,x^{-1}]$ is hereditarily atomic. Hence $R$ is an algebraic extension of $\ff_p$, which concludes the proof.
\end{proof}

In contrast to the class of polynomial rings and the class of Laurent polynomial rings, which contain many hereditarily atomic domains (see Example~ \ref{ex:rings of polynomials over Z are HAD} and Proposition~\ref{prop:polynomial rings that are HAD}), the class of power series rings does not contain any hereditarily atomic domain. We conclude the paper arguing this observation.

\begin{prop} \label{prop:power series rings are not HADs}
	If $R$ is an integral domain, then $R\ldb x \rdb$ is not hereditarily atomic.
\end{prop}

\begin{proof}
	Suppose, by way of contradiction, that $R\ldb x \rdb$ is hereditarily atomic. Let $R'$ be the prime subring of $R$. Since $R'$ is either finite or countable, the polynomial ring $R'[X]$ is also countable. Therefore the algebraic closure of $R'$ inside $R\ldb x \rdb$ is countable. On the other hand, it is clear that $R\ldb x \rdb$ has uncountably many units. As a result, there must be a unit $u$ of $R \ldb x \rdb$ that is not algebraic over the ring $R'[x]$. Then $R'[u,u^{-1}]$ is a subring of $R \ldb x \rdb$ that is isomorphic to the ring of Laurent polynomials in one variable over $R'$. It follows from Theorem~\ref{thm:Laurent polynomial rings} that $R'$ is an algebraic extension of $\ff_p$, and so $R' = \ff_p$ because $R'$ is the prime subring of $R$. Since $\ff_p[u, u^{-1}]$ is a countable set, we can argue as we did before to deduce that there is an element $y \in R \ldb x \rdb$ that is transcendental over the ring $\ff_p[u,u^{-1}]$. This clearly implies that $y$ is transcendental over $\ff_p$ and that both $u$ and $u^{-1}$ are transcendental over $\ff_p[y]$. Since $\ff_p[y]$ is not an algebraic extension of $\ff_p$, it follows from Theorem~\ref{thm:Laurent polynomial rings} that the Laurent polynomial ring $\ff_p[y][u,u^{-1}]$ is not hereditarily atomic. Hence we conclude that $R \ldb x \rdb$ is not hereditarily atomic.
\end{proof}

In contrast to Corollary~ \ref{cor: Noetherian/Krull domains HAD}, now we can identify Noetherian and Krull domains of any finite Krull dimension that are hereditarily atomic.
	
\begin{cor} \label{cor: Noetherian/Krull domains not HAD}
	For every $k \in \nn$, there exists a Noetherian/Krull domain of dimension $k$ that is not hereditarily atomic.
\end{cor}

\begin{proof}
	For each $k \in \nn$, the $k$-dimensional ring $R := \qq\ldb x_1, \dots, x_k \rdb$ is both a Noetherian domain and a Krull domain \cite[Theorem 12.4(iii)]{hM86}. In addition, $\qq\ldb x \rdb$ is a subring of $R$ that is not hereditarily atomic by Proposition~\ref{prop:power series rings are not HADs}. Thus, $R$ is not hereditarily atomic.
\end{proof}

\bigskip
\section{Hereditary Atomicity VS Hereditary ACCP}
\label{sec:hereditary ACCP}
\smallskip

It is well known (but historically a point of contention) that the ACCP condition is strictly stronger than the atomic condition. The first example illustrating this fact is due to Grams~\cite{aG74} (correcting an error in~\cite{pC68}). Further constructions have been given since then, and references to such constructions, including two recent ones using pullbacks of rings and monoid algebras, were pointed out in the introduction.

We say that an integral domain $R$ is \emph{hereditarily ACCP} if every subring of $R$ satisfies ACCP. It is clear that if $R$ is hereditarily ACCP, then it must be hereditarily atomic. Given the history of these concepts, it is natural to ask if the class consisting of hereditarily atomic domains coincides with that  of hereditarily ACCP domains. This motivates the following conjecture.

\begin{conj}\label{conjecture}
	For an integral domain $R$, the following conditions are equivalent.
	\begin{enumerate}
		\item[(a)] $R$ is hereditarily atomic.
		\smallskip
		
		\item[(b)] $R$ is hereditarily ACCP.
	\end{enumerate}
\end{conj}

We remark that this conjecture is not without historical motivation. For example, if $R\subseteq T$ is an extension of rings, then it is known that each of the first two conditions below implies the next one:
\begin{enumerate}
	\item[(a)] $R\subseteq T$ is a going up (GU) extension;
	\smallskip
	
	\item[(b)] $R\subseteq T$ is a lying over (LO) extension;
	\smallskip
	
	\item[(c)] $R\subseteq T$ is a survival extension.
\end{enumerate}
In addition, no two of the previous three conditions are equivalent. On the other hand, if we replace ``extension" with ``pair" in the same three conditions, then they become equivalent statements (we say that $R\subseteq T$ is an $X$-\emph{pair} if for any intermediate extensions $R\subseteq A\subseteq B\subseteq T$, the inclusion $A\subseteq B$ is an $X$-extension). The interested reader is directed to~\cite{CD03} and~\cite{D81} for more details.

Although Conjecture~\ref{conjecture} is still open, in this section we will reduce it to some specialized cases. We outline an approach that profiles the general properties of a canonical counterexample given that any exists. 
It is interesting to notice that the veracity of Conjecture \ref{conjecture} would show that any example of an atomic domain that is not ACCP must contain a non-atomic subring (and if Conjecture \ref{conjecture} proves false, this shows exactly where to look for examples).

\begin{prop}
The following conditions are equivalent.
\begin{enumerate}
	\item[(a)] Conjecture~\ref{conjecture} holds.
	\smallskip
	
	\item[(b)] Every integral domain that does not satisfy ACCP contains a subring that is not atomic.
\end{enumerate}
\end{prop}

\begin{proof}
	(a) $\Rightarrow$ (b): Assume Conjecture~\ref{conjecture} is true, and let $R$ be an integral domain that does not satisfy ACCP. Clearly, $R$ is not hereditarily ACCP and, under our assumption, this means that~$R$ is not hereditarily atomic.
	\smallskip
	
	(b) $\Rightarrow$ (a): If every integral domain without the ACCP condition contains a non-atomic subring, then no hereditarily atomic domain can have a subring without the ACCP condition.
\end{proof}

Here is a slight variant of \cite[Proposition~2.1]{aG74} (the original result by Grams  appears previously as our Proposition \ref{prop:ACCP subrings}).

\begin{prop}\label{augACCPsubrings}
	Let $R$ be an integral domain satisfying ACCP. If a subring~$S$ of $R$ has the property that every unit of $R$ is integral over $S$, then $S$ also satisfies ACCP.
\end{prop}

\begin{proof}
	We only need to show that strict containments among principal ideals survive in $R$. To this end, suppose that $(a)\subsetneq (b)$ in $S$,  and take $c \in S$ with $a=bc$. If $c$ is a unit in $R$, then $c^{-1}$ must be integral over $S$ by hypothesis and, hence, a unit itself in~$S$ (this is because $c^{-1}$ is integral over $S$ if and only if $c^{-1}\in S[c]=S$), which is a contradiction. Hence the desired inclusion remains strict in~$R$.
\end{proof}

Let $R$ be an integral domain. For any ascending chain of ideals $(I_n)_{n \in \nn}$ of $R$, we say that $(I_n)_{n \in \nn}$ \emph{ascends} from~$I_1$. Recall that a multiplicative subset $S$ of $R$ is saturated if $S$ is a divisor-closed submonoid of $R^*$.
%

We are in a position to prove our final result.

\begin{theorem}
	Let $R$ be an integral domain with quotient field $K$. Let $Z$ be the prime subring of~$R$, and let~$Q$ be the quotient field of $Z$. For each of the following conditions, $R$ is hereditarily atomic if and only if $R$ is hereditarily ACCP.
	\begin{enumerate}
		\item There exists $u \in R^\times$ such that $u$ is not integral over $Z$, and there exists $x\in R$ such that $x$ is transcendental over $Q(u)$.
		\smallskip
		
		\item The field $K$ has positive characteristic, and it is algebraic over $Q$. 
		\smallskip
		
		\item  One of the following conditions hold:
		\begin{itemize}
			\item $K$ has positive characteristic, $R$ satisfies ACCP, and the transcendence degree of~$K$ over~$Q$ is not $1$;
			\smallskip
			
			\item $K$ has characteristic zero, $R$ satisfies ACCP, and the transcendence degree of~$K$ over~$Q$ is at least $2$.
		\end{itemize}
		\smallskip
		
		\item There exists a sequence $(y_n)_{n \in \nn}$ whose terms are nonunits in $R$ with an $x \in \bigcap_{n \in \nn} (y_1 \cdots y_n)$ that is transcendental over the field $Q(y_n \mid n \in \nn)$.
		\smallskip
		
		\item There exists a strictly ascending chain $(x_n)_{n \in \nn_0}$ of principal ideals such that $x_0$ is transcendental over the field $Q\big(\frac{x_n}{x_{n+1}} \mid n \in \nn_0 \big)$.
		\smallskip
		
		\item There exists a strictly ascending chain of principal ideals of $R$
		\[
			(x) \subsetneq \Big( \frac{x}{y_1} \Big) \subsetneq \cdots \subsetneq \Big( \frac{x}{y_1y_2\cdots y_n} \Big) \subsetneq\cdots
		\] 
		for some nonzero nonunit $x  \in R$ and a sequence $(y_n)_{n \in \nn}$ whose terms are nonzero nonunits in~$R$ and $x \in \bigcap_{n \in \nn} (y_1 \cdots y_n)$ such that $Z[y_n \mid n \in \nn] \bigcap P = (0)$, where $P\subset R$ is the ideal  $\displaystyle\Big( x,\frac{x}{y_1},\cdots,\frac{x}{y_1y_2\cdots y_n},\cdots\Big)$.
		\smallskip
		
		\item The multiplicative subset of $R$ defined by 
		\[
			\{z\in R^* \mid \text{$(z)=R$ or there is an infinite chain of distinct principal ideals ascending from }(z)\}
		\]
		is a saturated set that is not contained in $R^\times$.
		\color{black}
	\end{enumerate}
\end{theorem}

\begin{proof}
We remark at the outset that, as noted before, hereditarily ACCP implies hereditarily atomic. So for our seven items, we show hereditarily atomic implies hereditarily ACCP and the usual approach will be to show produce an non-atomic subring of a domain that is not ACCP.

	(1) Take $u\in R^\times$ such that $u$ is not integral over $Z$, and then select $x \in R$ such that~$x$ is transcendental over $Q(u)$. It suffices to show that $R$ has a non-atomic subring $S$. Towards this end, let~$S$ be the subring $Z[u^{-1}]+xZ[u,u^{-1}][x]$ of $R$. Since $x$ is transcendental over $Q(u)$, it functions as a polynomial variable. Proving that $S$ is not atomic amounts to showing that $x$ cannot be factored into irreducibles in $S$. First, note that $u^{-1}$ is not a unit of $Z[u^{-1}]$ as otherwise $f(u^{-1}) = u$ for some polynomial $f$ over $Z$, and so $u$ would be a root of the monic polynomial $x^{1 + \deg f} - x^{\deg f} f(x^{-1})$ with coefficients in $Z$, which is not possible because $u$ is not integral over $Z$. Hence $u^{-1} \notin S^\times$, and so the equality $x = u^{-1}(xu)$ ensures that $x$ is not irreducible. Because the multiplicative monoid consisting of all nonzero monomials of $S$ is a divisor-closed submonoid of $S^*$, every potential factorization of~$x$ in~$S$ must contain a factor of the form $cx$ for some $c \in Z[u,u^{-1}]$, which is not irreducible in $S$ because~$x$ is not irreducible. Hence $\mathsf{Z}_S(x)$ is empty.
%
\smallskip


(2) Set $p := \text{char}(K)$. As $K$ is algebraic over $\mathbb{F}_p$, every element of $R$ is algebraic over $\mathbb{F}_p$ and, as a result, $R$ is a field. Suppose now that $R$ is hereditary atomic. Since $R$ is a hereditarily atomic field, Corollary~\ref{por} guarantees that every subring of $R$ is a BFD and, therefore, satisfies ACCP. Hence $R$ is hereditarily ACCP.
\smallskip

(3) In the case of positive characteristic, we can assume, by using part~(2), that the transcendence degree of $K$ over $Q$ is at least $2$, which is the blanket assumption for the case of characteristic zero. Since $R$ satisfies ACCP, we note that if there exists $u \in R^\times$ such that $u$ is not integral over $Z$, then the fact that the transcendence degree of $K$ over $Q$ is at least $2$ assures us that condition~(1) holds, and so we are done. Hence we can further assume that every unit of $R$ is integral over~$Z$. Under this last assumption, Proposition \ref{augACCPsubrings} gives the desired result.
\smallskip

(4) Take $(y_n)_{n \in \nn}$ to be a sequence with terms in $R \setminus R^\times$, and then take $x \in \bigcap_{n \in \nn} (y_1 \cdots y_n)$ such that $x$ is transcendental over $Q(y_n \mid n \in \nn)$. We claim that $R$ is not hereditarily atomic. To verify our claim it suffices to check that the subring $S := Z[y_n \mid n \in \nn] + Rx$ of $R$ is not atomic. Since $x \in (y_1 y_2)$, it is not irreducible in $S$. As $x$ is transcendental over $Q(y_n \mid n \in \nn)$, we can proceed as we did in the proof of~(1) to conclude that $\mathsf{Z}_S(x)$ is empty. Thus, $S$ is not atomic.
\smallskip

(5) This condition is equivalent to condition~(4). Indeed, one can readily show that $(x_n)_{n \in \nn_0}$ is a strictly ascending chain of principal ideals of $R$ ascending from $(x)$ if and only if the sequence $(y_n)_{n \in \nn}$ defined by $y_n := \frac{x_{n-1}}{x_n}$ consists of nonunits of $R$ and satisfies that $x \in \bigcap_{n \in \nn} (y_1 \cdots y_n)$.
\smallskip

(6) We consider here the subring $T:=Z[y_n \mid n \in \nn]+ P$. Take $r_1, r_2\in Z[y_n \mid n \in \nn]$ and $p_1,p_2\in P$ such that $(r_1+p_1)(r_2+p_2)\in P$. Since $r_1 r_2 \in Z[y_n \mid n \in \nn] \bigcap P = (0)$, we see that $r_1 r_2 = 0$ and, therefore, either $r_1 + p_1$ or $r_2 + p_2$ belongs to $P$. Thus, $P$ is a prime ideal of the subring~$T$. We now note that $x$ cannot be factored into irreducibles in $T$ by observing that if $x = a_1 a_2 \cdots a_k$ with $a_1, a_2, \dots, a_k \in \ii(T)$, then as $P$ is prime, $a_j\in P$  for some $j \in \ldb 1,k \rdb$, and so 
\begin{equation} \label{eq:part (6)}
	a_j=r_0x+\sum_{k=1}^nr_k\frac{x}{y_1y_2\cdots y_k}
\end{equation}
for some choices of $r_0, \dots, r_n\in R$. We observe that~\eqref{eq:part (6)} shows that $y_{n+1}$ divides $a_j$, and this is our desired contradiction. We conclude that $T$ is not atomic.
\smallskip

(7) Let $S$ be the multiplicative subset of $R$ described in condition~(7). It is clear that $S$ is a multiplicative subset of $R$. Note that in this case, the ring $R$ itself can be easily shown to be non-atomic. Indeed, suppose that $s \in S$ is a nonunit and write $s = r_1 r_2 \cdots r_k$ for some nonunits $r_1, r_2, \dots, r_k \in R$. As $S$ is saturated, each~$r_i$ has an infinite chain of principal ideals ascending from it and hence cannot be irreducible.
\end{proof}

Despite this run of cases, Conjecture~\ref{conjecture} is still open. It would be interesting to see an answer to this conjecture as it should provide a further step towards unraveling the subtle interplay between atomicity and the ACCP condition.

\bigskip
\section*{Acknowledgments}

While working on this paper, the second author was kindly supported by the NSF award DMS-1903069.

\bigskip


\begin{thebibliography}{20}
	
	\bibitem{AAZ90} D.~D. Anderson, D.~F. Anderson, and M. Zafrullah: \emph{Factorizations in integral domains}, J. Pure Appl. Algebra {\bf 69} (1990) 1--19.

	\bibitem{ACHZ07} D.~D.~Anderson, J. Coykendall, L. Hill, and M. Zafrullah: \emph{Monoid domain constructions of antimatter domains}, Comm. Algebra \textbf{35} (2007) 3236--3241.
	
	\bibitem{AJ15} D.~D. Anderson and J.~R. Juett: \emph{Long length functions}, J. Algebra \textbf{426} (2015) 327--343.
	
	\bibitem{AM96} D.~D. Anderson and B. Mullins: \emph{Finite factorization domains}, Proc. Amer. Math. Soc. \textbf{124} (1996) 389--396.
	
	\bibitem{AeA99} D.~F. Anderson and D.~N. El~Abidine: \emph{Factorization in integral domains III}, J. Pure Appl. Algebra {\bf 135} (1999) 107--127.
	
	\bibitem{AG21} D. F. Anderson and F. Gotti: \emph{Bounded and finite factorization domains}. In: Rings, Monoids, and Module Theory (Eds. A. Badawi and J. Coykendall) Springer Proceedings in Mathematics \& Statistics, vol. \textbf{382}, Singapore, 2022.
	
	\bibitem{BC19} J. G. Boynton and J. Coykendall: \emph{An example of an atomic pullback without the ACCP}, J. Pure Appl. Algebra \textbf{223} (2019) 619--625.
	
	\bibitem{BR76} J. Brewer and E. A. Rutter: \emph{$D+M$ constructions with general overrings}, Michigan Math. J. \textbf{23} (1976) 33--42.
	
	\bibitem{pC68} P. M. Cohn: \emph{Bezout rings and and their subrings}, Proc. Cambridge Philos. Soc. \textbf{64} (1968) 251--264.

	\bibitem{CDM99} J. Coykendall, D.~E. Dobbs, and B. Mullins: \emph{On integral domains with no atoms}, Comm. Algebra \textbf{27} (1999) 5813--5831.
	
	\bibitem{CD03} J. Coykendall and D.~E. Dobbs: \emph{Survival-pairs of commutative rings have the lying-over property}, Comm. Algebra \textbf{31} (2003) 259--270.
	
	\bibitem{CG19} J. Coykendall and F. Gotti: \emph{On the atomicity of monoid algebras}, J. Algebra \textbf{539} (2019) 138--151.
	
	\bibitem{D81} D.~E. Dobbs: \emph{Lying-over pairs of commutative rings}, Canadian Math. J. \textbf{33} (1981) 454--475.
	
	\bibitem{FHP97} M. Fontana, J.~A. Huckaba, and I.~J. Papick: \emph{Pr\"ufer domains}, Monographs and Textbooks in Pure and Applied Mathematics \textbf{203}, M. Dekker, New York, 1997.
	
	\bibitem{rF73} R. M. Fossum: \emph{The Divisor Class Group of a Krull Domain}, Springer-Verlag, New York, 1973.
	
	\bibitem{lF70} L. Fuchs: \emph{Infinite Abelian Groups I}, Academic Press, 1970.
	
	\bibitem{FS01} L. Fuchs and L. Salce: \emph{Modules over non-Noetherian Domains}, Mathematical Surveys and Monographs, vol. 84, Providence, 2001.
	
	\bibitem{GH06} A. Geroldinger and F. Halter-Koch: \emph{Non-unique Factorizations: Algebraic, Combinatorial and Analytic Theory}, Pure and Applied Mathematics Vol. 278, Chapman \& Hall/CRC, Boca Raton, 2006.
	
	\bibitem{rG84} R. Gilmer, Commutative Semigroup Rings, The University of Chicago Press, 1984.
	
	\bibitem{rG64} R. Gilmer: \emph{Integral domains which are almost Dedekind}, Proc. Amer. Math. Soc. \textbf{15} (1964) 813--818.
	
	\bibitem{rG70} R. Gilmer: \emph{Integral domains with Noetherian subrings}, Comment. Math. Helv. \textbf{45} (1970) 129--134.
	
	\bibitem{rG68} R. Gilmer: \emph{Multiplicative Ideal Theory}, Queen’s Papers in Pure and Applied Mathematics, No. 12, University of Kingston, Ontario, 1968.
	
	
	\bibitem{rG83} R. Gilmer: \emph{Noetherian pairs adn hereditarily Noetherian rings}, Arch. Math. \textbf{41} (1983) 131--138.
	
	\bibitem{GP74} R. Gilmer and T. Parker: \emph{Semigroup rings as Pr\"ufer rings}, Duke Math. J. \textbf{41} (1974) 219--230.
	
	\bibitem{fG19} F. Gotti: \emph{Increasing positive monoids of ordered fields are FF-monoids}, J. Algebra \textbf{518} (2019) 40--56.
	
	\bibitem{fG21} F. Gotti: \emph{On semigroup algebras with rational exponents}, Communications in Algebra (to appear). DOI: https://doi.org/10.1080/00927872.2021.1949018.
	
	\bibitem{GL21} F. Gotti and B. Li: \emph{Atomic semigroup rings and the ascending chain condition on principal ideals}. Preprint on arXiv: https://arxiv.org/pdf/2111.00170.pdf 
	
	\bibitem{aG74} A. Grams: \emph{Atomic rings and the ascending chain condition for principal ideals}, Proc. Cambridge Philos. Soc., \textbf{75} (1974) 321--329.
	
	\bibitem{HK92} F. Halter-Koch: \emph{Finiteness theorems for factorizations}, Semigroup Forum \textbf{44} (1992) 112--117.
	
	\bibitem{iK74} I. Kaplansky: \emph{Commutative Rings}, The University of Chicago Press, Chicago, Ill.-London, revised edition, 1974.
	
	\bibitem{hK01} H. Kim: \emph{Factorization in monoid domains}, Comm. Algebra \textbf{29} (2001) 1853--1869.
	
	\bibitem{LM71} M.~D. Larsen and P.~J. McCarthy: \emph{Multiplicative Theory of Ideals}, Academic Press, New York and London, 1971.
	
	\bibitem{hM86} H. Matsumura: \emph{Commutative Ring Theory}, Cambridge studies in advanced mathematics, Vol. 8, Cambridge University Press, 1986.
	
	\bibitem{nN53} N. Nakano: \emph{Ideal theorie in einem speziellen unindlichen algebraishen Zahlk\"oper}, J. Sci. Hiroshima Univ. Ser. A \textbf{16} (1953) 425--439.
	
	 \bibitem{jO66} J. Ohm: \emph{Some counterexamples related to integral closure in $D[[x]]$}, Trans. Amer. Math.
	Soc. \textbf{122} (1966) 321--333.
	
	\bibitem{rP63} R. C. Phillips: \emph{Almost Dedekind domains}, PhD Dissertation, Louisiana State University, Baton Rouge, 1963.
	
	\bibitem{mR93} M. Roitman: \emph{Polynomial extensions of atomic domains}, J. Pure Appl. Algebra \textbf{87} (1993) 187--199.
	
	\bibitem{aZ82} A. Zaks: \emph{Atomic rings without a.c.c. on principal ideals}, J. Algebra \textbf{74} (1982) 223--231.
\end{thebibliography}
\end{document}